\documentclass[11pt,a4paper,reqno]{amsart}

\usepackage{epsfig}
\usepackage{amsfonts}
\usepackage{amsmath}
\usepackage{amssymb}
\usepackage{amsthm}
\usepackage{mathtools}
\usepackage{color}
\usepackage{enumerate}
\usepackage[T1]{fontenc}
\usepackage[latin1]{inputenc}
\usepackage{ae,aecompl}
\usepackage{graphicx}
\usepackage{nicefrac}
\usepackage[includehead,includefoot,margin=25mm]{geometry}
\usepackage{faktor}
\usepackage{dsfont}
\usepackage[all]{xy}
\usepackage{mathrsfs}
\usepackage{textcomp}
\usepackage{cite}

\newtheorem*{theorem*}{Theorem}

\newtheorem*{lemma*}{Lemma}

\newtheorem*{proposition*}{Proposition}

{\theoremstyle{remark}
  
  \newtheorem*{remark*}{Remark}}
{\theoremstyle{definition}

  \newtheorem*{example*}{Example}

}

\newcommand{\CC}[0]{\ensuremath{\mathbb{C}}}
\newcommand{\ZZ}[0]{\ensuremath{\mathbb{Z}}}

\newcommand{\GM}[0]{\ensuremath{\mathbb{G}_{\mathrm{m}}}}
\newcommand{\AF}[0]{\ensuremath{\mathbb{A}}}

\newcommand{\QQ}[0]{\ensuremath{\mathbb{Q}}}
\newcommand{\TT}[0]{\ensuremath{\mathrm{T}}}
\newcommand{\KK}[0]{\ensuremath{\mathbb{C}}}

\newcommand{\spec}[0]{\ensuremath{\operatorname{Spec}}}

\begin{document}

\title{Uniformly rational varieties with torus action}

\author{Alvaro Liendo} %
\address{Instituto de Matem\'atica y F\'\i sica, Universidad de Talca,
  Casilla 721, Talca, Chile.}%
\email{aliendo@inst-mat.utalca.cl}

\author{Charlie Petitjean} %
\address{Instituto de Matem\'atica y F\'\i sica, Universidad de Talca,
  Casilla 721, Talca, Chile.}%
\email{petitjean.charlie@gmail.com}

\date{\today}

\thanks{
  \mbox{\hspace{11pt}}{\it Key words}: Uniformly rational varieties,
  T-varieties, torus actions, algebraic quotient.\\
  \mbox{\hspace{11pt}}This research was partially supported by
  projects Fondecyt regular 1160864 and Fondecyt postdoctorado
  3160005}

\begin{abstract}
  A smooth variety is called uniformly rational if every point admits
  a Zariski open neighborhood isomorphic to a Zariski open subset of
  the affine space. In this note we show that every smooth and
  rational affine variety endowed with an algebraic torus action such
  that the algebraic quotient has dimension 0 or 1 is uniformly
  rational.
\end{abstract}

\maketitle

\section*{Introduction}

In the seminal paper \cite{G}, Gromov discusses the class of algebraic
varieties $X$ having the property that every point admits a Zariski
open neighborhood isomorphic to a Zariski open subset of the affine
space. These varieties are now called uniformly rational varieties
\cite{B-Bo}. An uniformly rational variety is clearly smooth and
rational. Furthermore, in dimensions $1$ and $2$, every smooth and
rational variety is uniformly rational. In higher dimension, it is an
open question whether all smooth and rational varieties are uniformly
rational.

Several families of uniformly rational varieties are known. For
instance, rational varieties that are homogeneous with respect to the
action of an algebraic group are uniformly rational. Smooth toric
varieties are also uniformly rational. Furthermore, blow-ups along
smooth centers of uniformly rational varieties are uniformly rational.

In this note we deal with smooth varieties endowed with torus actions
that are not necessarily toric. We work over the field of complex
numbers $\KK$. Let $\TT$ be the $n$-dimensional algebraic torus
$\GM^n$, where $\GM$ is the multiplicative group. The complexity of a
$\TT$-action is the codimension of a generic orbit. Since the quotient
of an algebraic torus by a normal algebraic subgroup is again an
algebraic torus, up to changing the torus, we can always assume that a
$\TT$-action is faithful. A $\TT$-variety is a normal variety endowed
with a faithful $\TT$-action. Hence, toric varieties correspond to
$\TT$-varieties of complexity zero.

As stated before, all smooth toric varieties are uniformly
rational. This follows from the fact that every smooth toric variety
without torus factor admits a $\TT$-equivariant open cover by affine
charts isomorphic to the affine space. Furthermore, it follows in a
straightforward way from \cite[Chapter 4]{KKMSD} that smooth and
rational $\TT$-varieties of complexity one are also uniformly
rational, Moreover, if $X$ is also complete then $X$ also admits a
admits a $\TT$-equivariant open cover by affine charts isomorphic to
the affine space, see \cite{A-Pe-S}.

In higher complexity the situation is less clear. It is not known
whether all smooth and rational $\TT$-varieties of complexity two or
higher are uniformly rational, but the second author has provided
counterexamples to an equivariant version of uniform rationality. A
$G$-variety $X$ is said to be equivariantly uniformly rational if it
admits an $G$-invariant open cover by open sets $G$-equivariantly
isomorphic to $G$-invariant open sets of the affine space endowed with
a $G$-action. In \cite{P}, the author has given examples of affine
$\TT$-varieties of complexity two and higher that are not
equivariantly uniformly rational.

Nevertheless, we prove that a large class of affine $\TT$-varieties
are uniformly rational. Indeed, we show that a smooth and rational
affine $\TT$-variety $X$ of any complexity is uniformly rational,
provided that the algebraic quotient is of dimension at most
one. Recall that the algebraic quotient
$X/\!/\TT:=\spec\left(\KK[X]^\TT\right)$ is the spectrum of the ring
of invariant regular function. More precisely, we will prove the
following theorem.

\begin{theorem*} Let $X$ be a smooth and rational affine
  $\TT$-variety.
  \begin{enumerate}
  \item If the algebraic quotient $X/\!/\TT$ is a point then $X$ is
    equivariantly isomorphic to $(\KK^{*})^{l}\times\AF^{n-l}$. In
    particular, $X$ is uniformly rational.
  \item If the algebraic quotient $X/\!/\TT$ is a curve then $X$ is
    uniformly rational.
  \end{enumerate}
\end{theorem*}

\section*{Proof of the result}

Let $N\simeq\ZZ^k$ be a lattice of rank $k$, and let
$M=\operatorname{Hom}(N,\ZZ)$ be its dual lattice. We let
$\TT=\spec(\KK[M])$ be the algebraic torus of dimension $k$. This
ensures that $M$ is character lattice of $\TT$ and $N$ is the
$1$-parameter subgroup lattice of $\TT$. We also let $M_\QQ$ be the
$\QQ$-vector space $M\otimes_\ZZ\QQ$ and $N_\QQ$ be the $\QQ$-vector
space $N\otimes_\ZZ\QQ$. There is a natural duality pairing
$\langle\cdot,\cdot\rangle:M_\QQ\times N_\QQ\rightarrow \QQ$.

We consider now an algebraic affine $\TT$-variety $X=\spec A$. The
$\TT$-action on $X$ corresponds to an $M$-grading of the ring $A$ of
regular functions where for every $u\in M$, the homogeneous piece
$A_u\subset A$ is given by the semi-invariant functions with respect
to the character $\chi^{u}$, i.e.,
$$A=\bigoplus_{u\in M} A_u\qquad\mbox{with}\qquad A_u=
\{f\in A\mid \lambda.f=\chi^u(\lambda)\cdot f\}\,.$$

The weight monoid $S$ attached to the $\TT$-action on $X$ corresponds
to all element $u\in M$ such that $A_{u}\neq\{0\}$. The cone spanned
by $S$ in the vector space $M_\QQ$ is called the weight cone of the
$\TT$-action and we denote it by $\omega$.

Our proof of the theorem is a consequence of a canonical factorization
of the quotient morphism $\pi:X\rightarrow X/\!/\TT=\spec A_0$ that
may be of independent interest. We state this factorization in the
following proposition. First, we need some definitions.

A $\TT$-action is said to be fix-pointed if the only vector space
contained in the weight cone $\omega$ is $\{0\}$. This is the case if
and only if algebraic quotient $X/\!/\TT$ is isomorphic to the fixed
point locus $X^{\TT}$ via the composition
$X^\TT\xhookrightarrow{\hspace*{1.5em}}
X\stackrel{\pi}{\xrightarrow{\hspace*{1.5em}}} X/\!/\TT$.
A $\TT$-action is said to be hyperbolic if the weight cone $\omega$ is
the whole $M_\QQ$. This is the case if and only if the dimension of
the algebraic quotient coincides with the complexity of the
$\TT$-action.

\begin{proposition*}
  For every $\TT$-variety $X$ there is an unique splitting of the
  acting torus $\TT=\TT_1\times \TT_2$ such that
  \begin{enumerate}
  \item The action of $\TT_1$ on $X$ is fix-pointed.
  \item The torus $\TT_2$ acts (non necessarily faithfuly) on
    $X_H:=X/\!/\TT_1$ and this action is hyperbolic.
  \item The quotient morphism $\pi:X\rightarrow X/\!/\TT$ factorizes
    as
    \[ X\stackrel{/\!/\TT_{1}}{\xrightarrow{\hspace*{3em}}} X_H
    \stackrel{/\!/\TT_{2}}{\xrightarrow{\hspace*{3em}}}
    X_H/\!/\TT_2\simeq X/\!/\TT\,.
    \]
  \end{enumerate}
\end{proposition*}

\begin{proof}
  Let $H$ be the biggest vector space contained in
  $\omega\subseteq M_\QQ$. We define the finitely generated and graded
  algebra $A_{H}=\bigoplus_{m\in H\cap M}A_{u}$ which gives us the
  algebraic variety $X_{H}:=\spec A_{H}$.  This yields a sequence of
  inclusion of algebras
  \[ A \xhookleftarrow{\hspace*{3em}} A_H \xhookleftarrow{\hspace*{3em}}
  A_0\,,
  \]
  where $A_{0}$ is the ring of invariant functions for the
  $\TT$-action and so is the algebra of functions of the algebraic
  quotient $X/\!/\TT$.

  Let $N_1=H^\bot\cap N$ be the sublattice of $N$ of vectors
  orthogonal to $H$. By definition, $N_1$ is saturated as
  sublattice. Then the quotient $N/N_{1}$ is torsion free which by
  \cite[Exercice 1.3.5.]{Cox} implies the existence of a complementary
  sublattice $N_{2}\subseteq N$ such that $N=N_{1}\bigoplus
  N_{2}$.
  Let $M_{i}$ be the dual lattice of $N_{i}$ and
  $\TT_{i}=\spec \KK[M_{i}]$ the associated algebraic torus for
  $i=1,2$. This yields a splitting $\TT=\TT_1\times \TT_2$ of the
  original torus.

  The previous sequence of inclusion of algebra provides the following
  sequence of algebraic quotients
  \[ X\stackrel{/\!/\TT_{1}}{\xrightarrow{\hspace*{3em}}} X_H
  \stackrel{/\!/\TT_{2}}{\xrightarrow{\hspace*{3em}}}
  X_H/\!/\TT_2\simeq X/\!/\TT\,.
  \]
  By construction, the action of $\TT_{1}$ on $X$ is fix-pointed and
  faithful whereas the action of $\TT_{1}$ on $X_{H}$ is
  hyperbolic. The uniqueness of such splitting is clear from the
  construction.
\end{proof}

In the proof of our theorem we need the following result directly
borrowed from \cite[Theorem~2.5]{Kam-R}. 

\begin{lemma*} With the above notation, if $X$ is smooth then $X_{H}$
  is smooth and $X$ admits a structure of a vector bundle over $X_{H}$
  where each fiber of the vector bundle is stable under the
  $\TT_{1}$-action and $\TT_{1}$ acts linearly on it.
\end{lemma*}

In \cite{Kam-R} a fix-pointed action is called unmixed. We prefer to
call it fix-pointed since it is a notion that makes sense for
algebraic groups different from the torus. Remark that part $(1)$ of
our theorem provides a slight generalization of the above lemma in the
case where $X/\!/\TT$ is reduced to a point.

\begin{proof}[Proof of the theorem]
  To prove $(1)$, remark that $X_{H}$ admits a hyperbolic
  $\TT_2$-action and by hypothesis we have
  $\dim(X/\!/\TT)=\dim(X_{H}/\!/\TT_{2})=0$. Since for a hyperbolic
  action the dimension of the algebraic quotient equals the complexity
  we have that the $\TT_2$-action is of complexity zero and so $X_{H}$
  is a toric variety.

  Since $X$ is smooth, we have that $X_{H}$ is smooth and $X$ has the
  structure of a vector bundle over $X_H$ by the lemma.  Hence $X_H$
  is isomorphic to $(\KK^{*})^{a}\times\AF^{b}$.  By the
  generalization of the Quillen-Suslin theorem \cite{Q,S} (see
  \cite{Swan} for the generalization) any vector bundle over a ring of
  Laurent polynomial is trivial.  Thus $X$ is a smooth toric variety
  isomorphic to $(\KK^*)^{l}\times\AF^{n-l}$.

  Let us now prove the second assertion of the theorem.  Since $X$ is
  rational, we have that $X/\!/\TT$ is unirational. By hypotesis,
  $X/\!/\TT$ is a curve so we obtain that $X/\!/\TT$ is a rational
  curve. This shows that $X_{H}$ is also a rational variety. 

  Furthermore, by the lemma we have that $X_H$ is smooth. Since
  $X_{H}$ admits an hyperbolic $\TT$-action and by hypothesis
  $\dim(X/\!/\TT)=\dim(X_{H}/\!/\TT_{1})=1$, with the same argument
  above we obtain that the $\TT_2$-action on $X_H$ is of complexity
  one. Then $X_{H}$ is uniformly rational by \cite[Chapter
  4]{KKMSD}. Finally, by the lemma $X$ is a vector bundle over $X_H$
  and any vector bundle over a uniformly rational variety is uniformly
  rational (see also \cite[Example~2.1]{B-Bo}).
\end{proof}

\begin{example*}
  Let $X$ be the hypersurface in $\AF^{5}$ given by
  $$X=\{(x,y,z,t,u)\in \AF^5 \mid zty+x^{2}+y+t^{2}u=0\}\,.$$
  Let $\TT=\GM^{2}$. The variety is $\TT$-stable for the
  linear $\TT$-action on $\AF^{5}$ given by
  $$(\lambda_{1},\lambda_{2})\cdot(x,y,z,t,u)=(\lambda_{1}\lambda_{2}x,\lambda_{1}^{2}\lambda_{2}^{2}y,\lambda_{1}z,\lambda_{1}^{-1}t,\lambda_{1}^{4}\lambda_{2}^{2}u)\,.$$
  Then $X$ is a smooth and rational $\TT$-variety of complexity
  two. The algebraic quotient $X/\!/\TT$ is $\AF^{1}=\spec
  \CC[zt]$. Hence, $X$ is uniformly rational by our Theorem.
\end{example*}

\begin{remark*} 
  In the situation where $\dim(X/\!/\TT)\geq2$ it was proven by the
  second author in \cite{P} that there exist smooth and rational
  affine threefolds endowed with an hyperbolic $\GM$-action that are
  not $\GM$-uniformly rational.  For instance the hypersurface $X$ in
  $\AF^{4}$ given by
  \[
  \{(x,y,z,t)\in \AF^4\mid z^{2}y+x^{3}y^{2}+x+t^{3}=0\}.
  \]
  admits a $\GM$-action obtained by restricting the linear
  $\GM$-action on $\AF^4$ given by
  $$\lambda\cdot(x,y,z,t)=(\lambda^{3}x,\lambda^{-3}y,\lambda^{3}z,\lambda
  t)\,.$$
\end{remark*}

\end{document}